\documentclass{article}
\usepackage[utf8]{inputenc}
\usepackage[margin=1in]{geometry}
\usepackage{graphicx}
\usepackage{amssymb}
\usepackage{amsthm}
\usepackage{epstopdf}
\usepackage{dsfont}
\usepackage{xcolor}
\usepackage{amsmath}
\usepackage{mathrsfs}
\usepackage[parfill]{parskip}
\usepackage{comment}
\usepackage{tikz}
\usepackage{tikz-3dplot}
\usetikzlibrary{intersections}
\usepackage{mathtools}
\usepackage{hyperref}
\usepackage{caption}
\usepackage{enumitem}
\newtheorem{theorem}{Theorem}[section]

\newtheorem{lemma}[theorem]{Lemma}

\theoremstyle{definition}
\newtheorem{definition}[theorem]{Definition}
\newtheorem{observation}[theorem]{Observation}

\newtheorem{problem}[theorem]{Problem}
\newtheorem{remark}[theorem]{Remark}
\newtheorem{claim}[theorem]{Claim}
\title{Reconstructing Polytopes and Pseudomanifolds}
\author{Joshua Hinman\\
\small Department of Mathematics\\
\small University of Washington\\
\small Seattle, WA 98195-4350, USA\\
\small \texttt{joshrh@uw.edu}}
\date{}

\begin{document}
\maketitle

\begin{abstract}
    We prove that every 4-polytope is determined by its edge-polygon incidences, solving an open problem of Gr\"unbaum. For each $d \geq 3$, we show that \emph{not} every $d$-polytope is determined by its $(d-3)$-skeleton and dual $(d-3)$-skeleton together, answering a question of Samper.

    In the simplicial realm, we prove that for $d \geq 4$ and $\lceil \frac{d}{2} \rceil \leq k \leq d-2$, every homology $(d-1)$-manifold is determined by the incidences of its $k$- and $(k-1)$-faces. For $d \geq 5$ and $\lceil \frac{d+1}{2} \rceil \leq k \leq d-2$, we extend our proof to normal $(d-1)$-pseudomanifolds whose $(2d-2k-1)$-dimensional links are homology manifolds. Finally, we prove that \emph{not} every normal $(d-1)$-pseudomanifold is determined by its $(d-2)$-skeleton.
\end{abstract} 

\section{Introduction}
This paper is dedicated to reconstruction problems on polytopes and simplicial complexes. Reconstruction problems formalize the following question: if $P$ is an unknown $d$-polytope, how much of the face poset of $P$ must we know to reconstruct the rest?

The \emph{face poset} $\mathscr{F}(P)$ of a $d$-polytope $P$ is the set of faces of $P$, partially ordered by containment. Many reconstruction problems begin with the restriction of $\mathscr{F}(P)$ to faces of certain dimensions. One such problem is to reconstruct $P$ from its \emph{$k$-skeleton}, defined for $0 \leq k < d$ as the restriction of $\mathscr{F}(P)$ to faces of dimensions $0, \ldots, k$.

These problems are long-standing. In 1932, Whitney proved that any 3-polytope can be recovered from its vertex-edge graph. Specifically, Whitney showed that if $G$ is the vertex-edge graph of a 3-polytope, then the boundaries of facets are exactly the induced, non-separating cycles in $G$. \cite{whitney32}

Extending Whitney's result, Gr\"unbaum proved in 1967 that every $d$-polytope is determined by its $(d-2)$-skeleton. Gr\"unbaum's method paralleled Whitney's: if $\mathscr{C}$ is the $(d-2)$-skeleton of a $d$-polytope, Gr\"unbaum showed that the boundaries of facets are exactly the induced, non-separating subcomplexes of $\mathscr{C}$ homeomorphic to $S^{d-2}$. Gr\"unbaum noted that \emph{not} every $d$-polytope is determined by its $(d-3)$-skeleton, giving the counterexample of two distinct 4-polytopes with vertex-edge graph $K_8$. {\cite[pp.~228-229]{grunbaum03}}

Two other results of Gr\"unbaum are as follows. First, if $P$ is a $d$-polytope and $0 \leq a < b < d$, then the incidences between $a$- and $b$-faces of $P$ uniquely determine the restriction of $\mathscr{F}(P)$ to dimensions $a, \ldots, b$ {\cite[pp.~233]{grunbaum03}}. Effectively, given any two ``levels" of a face poset, we can recover everything in between.

Second, if $d \geq 5$, the restriction of $\mathscr{F}(P)$ to dimensions $1, \ldots, d-2$ is enough to reconstruct $P$ {\cite[pp.~234]{grunbaum03}}. Gr\"unbaum states that the $d=4$ analogue is an open problem:

\begin{problem}[Gr\"unbaum]
    \label{edge-ridge-problem}
    Are all 4-polytopes determined by their edge-ridge incidences?
\end{problem}

Problem \ref{edge-ridge-problem} is the final piece in a large puzzle. If $P$ is an arbitrary $d$-polytope, and we are given the restriction of $\mathscr{F}(P)$ to certain dimensions, then which dimensions must be included to guarantee we can reconstruct $P$? The answer is trivial for $d=2$: either the vertex set or the edge set is enough. Whitney's theorem covers $d=3$: we can reconstruct $P$ from its vertex-edge graph or, dually, its facet-ridge graph. For $d \geq 5$, Gr\"unbaum's results imply the following:
\begin{itemize}
    \item We can reconstruct $P$ if our subposet includes vertices or edges \emph{and} includes facets or ridges.
    \item We cannot necessarily reconstruct $P$ otherwise, as neither the $(d-3)$-skeleton nor the dual $(d-3)$-skeleton is sufficient in general.
\end{itemize}
Only for $d=4$ has the answer remained unknown.

A more general question arises naturally: if we know \emph{several} subposets of $\mathscr{F}(P)$, each being the restriction of $\mathscr{F}(P)$ to a distinct set of dimensions, when are these enough to reconstruct $P$? Samper recently asked a question of this type \cite{samper25}:

\begin{problem}[Samper]
    \label{samper-problem}
    Is every 4-polytope determined by the combination of its vertex-edge graph and its facet-ridge graph?
\end{problem}

Another genre of reconstruction problem deals with simplicial polytopes. If we know that $P$ is a simplicial $d$-polytope, then we can reconstruct $P$ from relatively little information, as the possible face posets of simplicial polytopes are far more restricted than those of general polytopes.

One celebrated theorem is attributed to Perles from 1970: any simplicial $d$-polytope is determined by its $\lfloor \frac{d}{2} \rfloor$-skeleton \cite{perles70}. In 1984, Dancis extended this theorem to homology $(d-1)$-manifolds for even $d$, and to homology $(d-1)$-manifolds with trivial $\frac{d-1}{2}$-homology for odd $d$. Dancis further proved that all homology $(d-1)$-manifolds are determined by their $\lceil \frac{d}{2} \rceil$-skeleta \cite{dancis84}.

In 1987, Blind and Mani proved that any simplicial polytope is determined by its facet-ridge graph \cite{blind87}. Kalai published a simpler proof the following year. Kalai worked in the dual setting, giving an algorithm to reconstruct a simple polytope from its vertex-edge graph \cite{kalai88}. By an additional result of Kalai, for any simplicial $d$-polytope $P$ and $1 \leq k < d$, the $k$-skeleton of $P$ is determined by the incidences of its $k$- and $(k-1)$-faces. Putting this all together, for any $\lfloor \frac{d}{2} \rfloor \leq k < d$, we can reconstruct $P$ from its $k$-, $(k-1)$-face incidences (see {\cite[Theorem 19.5.30]{kalai97}}).

This paper builds on the results above, both for polytopes and for simplicial complexes. We answer Problem \ref{edge-ridge-problem} in the affirmative, proving that all 4-polytopes are determined by their edge-ridge incidences (Theorem \ref{edge_ridge}). Our method involves Alexander duality, which we use to prove a ``non-separating" condition on the boundaries of facets akin to the conditions of Whitney and Gr\"unbaum. We then answer Problem \ref{samper-problem} in the negative by constructing distinct 4-polytopes with the same vertex-edge graph and the same facet-ridge graph. More generally, for each $d \geq 3$, we construct distinct $d$-polytopes with identical $(d-3)$-skeleta and identical dual $(d-3)$-skeleta (Theorem \ref{identical_skeleta}).

Next, we partially generalize Kalai and Perles's results on $k$-, $(k-1)$-face incidences from simplicial polytopes to pseudomanifolds. We prove that any simplicial homology $(d-1)$-manifold with $d \geq 4$ is determined by its $k$-, $(k-1)$-face incidences for arbitrary $\lceil \frac{d}{2} \rceil \leq k \leq d-2$ (Theorem \ref{half_dimension}). We further extend this to normal pseudomanifolds in which each $(2d-2k-1)$-dimensional link is a homology manifold (Theorem \ref{reconstruct_pseudomanifold}). Finally, we show that the reconstruction results for simplicial polytopes and manifolds fail for normal pseudomanifolds in general, as for each $d \geq 3$ we construct distinct, normal $(d-1)$-pseudomanifolds with identical $(d-2)$-skeleta (Theorem \ref{design-theorem}).

The structure of our paper is as follows. In Section \ref{s:preliminaries}, we define our objects of study and recall existing reconstruction results. In Section \ref{s:4-polytopes}, we discuss 4-polytopes and solve Problem \ref{edge-ridge-problem}. Section \ref{s:wedges} is home to our counterexamples for Problem \ref{samper-problem} and its higher-dimensional analogues. In Section \ref{s:main-result}, we combine our preceding results with those of Gr\"unbaum into a single reconstruction theorem, characterizing the incidence information needed to identify arbitrary $d$-polytopes for all $d \geq 4$. We conclude in Section \ref{s:simplicial} with our results for reconstructing simplicial complexes.

\section{Preliminaries}
\label{s:preliminaries}
In this section, we introduce the objects and tools involved in reconstruction problems. The reader may consult Gr\"unbaum \cite{grunbaum03} or Ziegler \cite{ziegler95} for any undefined terminology. The reader may also refer to Bayer \cite{bayer18} for a broader survey of existing reconstruction results.

\subsection{Polytopes and skeleta}
\begin{definition}
    For $d \geq 0$, a \emph{$d$-polytope} is the convex hull of a finite point set in $\mathds{R}^d$ whose affine hull is $\mathds{R}^d$.
\end{definition}

\begin{definition}
    Let $P$ be a $d$-polytope. A \emph{face} of $P$ is either $P$ itself or some $P \cap H$, where $H \subset \mathds{R}^d$ is a codimension one hyperplane that does not intersect the interior of $P$. For $0 \leq k < d$, a \emph{$k$-face} of $P$ is a face with a $k$-dimensional affine hull; a 0-, 1-, $(d-2)$-, or $(d-1)$-face of $P$ is called a \emph{vertex}, \emph{edge}, \emph{ridge}, or \emph{facet}, respectively.
\end{definition}

\begin{definition}
    Let $P$ be a $d$-polytope. The \emph{face poset} $\mathscr{F}(P)$ is the set of faces of $P$, partially ordered by containment. We denote by $\mathscr{F}^k(P)$ the set of $k$-faces of $P$.
\end{definition}

We may now state our reconstruction problem more precisely. Let $P$ be an unknown $d$-polytope. Suppose we know, up to isomorphism, some induced subposets of $\mathscr{F}(P)$; each subposet has the form $\bigcup_{k \in K} \mathscr{F}^k(P)$ for some $K \subseteq \{0,\ldots,d-1\}$. From which subposets of this kind is it guaranteed that we can reconstruct $\mathscr{F}(P)$?

This is a broad question, but we can simplify it using the notion of an \emph{$(a,b)$-skeleton} and a theorem of Gr\"unbaum.

\begin{definition}
    Let $P$ be a $d$-polytope and $0 \leq a \leq b < d$. We define the \emph{$(a,b)$-skeleton} of $P$ as the set
    \[
        \mathscr{F}_a^b(P) := \bigcup_{k=a}^b \mathscr{F}^k(P),
    \]
    with partial order induced by $\mathscr{F}(P)$. In particular, we call $\mathscr{F}_0^k(P)$ the \emph{$k$-skeleton} of $P$, and we call $\mathscr{F}_{d-k-1}^{d-1}(P)$ the \emph{dual $k$-skeleton} of $P$.
\end{definition}

\begin{theorem}[Gr\"unbaum {\cite[pp.~233]{grunbaum03}}]
    \label{interval}
    Let $P,Q$ be $d$-polytopes and $0 \leq a < b < d$. Then any poset isomorphism $\varphi:\mathscr{F}^a(P) \cup \mathscr{F}^b(P) \to \mathscr{F}^a(Q) \cup \mathscr{F}^b(Q)$ extends to an isomorphism $\mathscr{F}_a^b(P) \to \mathscr{F}_a^b(Q)$.
\end{theorem}

In other words, if $P$ is a $d$-polytope and $0 \leq a < b < d$, then the $(a,b)$-skeleton of $P$ is determined by $\mathscr{F}^a(P) \cup \mathscr{F}^b(P)$. Thus, our problem reduces to that of reconstructing a $d$-polytope from a set of $(a,b)$-skeleta.

\begin{problem}[Reconstruction]
\label{reconstruction_problem}
    Let $d \geq 2$. Characterize all
    \begin{gather*}
        0 \leq a_1 \leq b_1 < d,\\
        \vdots\\
        0 \leq a_n \leq b_n < d\phantom{,}
    \end{gather*}
    such that the face poset of every $d$-polytope $P$ is determined by the list of isomorphism types of $\mathscr{F}_{a_1}^{b_1}(P), \ldots,\allowbreak \mathscr{F}_{a_n}^{b_n}(P)$.
\end{problem}

Let us take this opportunity to answer Problem \ref{reconstruction_problem} for $d=2,3$.

\begin{observation}
\label{d=2}
    The face poset of any 2-polytope $P$ is determined by $\mathscr{F}^0(P)$, as well as by $\mathscr{F}^1(P)$. In other words, the number of vertices or edges in a polygon determines its face poset.
\end{observation}

\begin{observation}
\label{d=3}
    By Whitney's theorem \cite{whitney32}, the face poset of any 3-polytope $P$ is determined by $\mathscr{F}_0^1(P)$, as well as by $\mathscr{F}_1^2(P)$. Nothing else will suffice; for example, let $P$ be the cross-polytope and $Q$ the stacked polytope on six vertices. Then $f(P) = f(Q) = (6, 12, 8)$, so $\mathscr{F}^k(P) \cong \mathscr{F}^k(Q)$ for $k = 0,1,2$. However, $\mathscr{F}(P) \not\cong \mathscr{F}(Q)$.
\end{observation}

It remains to answer Problem \ref{reconstruction_problem} for $d \geq 4$. Two additional theorems of Gr\"unbaum take us much of the distance. The first says that we can reconstruct a $d$-polytope from its $(d-2)$-skeleton; the second says that if $d \geq 5$, then the $(1,d-2)$-skeleton is sufficient.

\begin{theorem}[Gr\"unbaum {\cite[pp.~228]{grunbaum03}}]
    \label{(d-2)-skeleton}
    Let $P,Q$ be $d$-polytopes with $d \geq 2$. Any isomorphism $\varphi:\mathscr{F}_0^{d-2}(P) \to \mathscr{F}_0^{d-2}(Q)$ extends to an isomorphism $\mathscr{F}(P) \to \mathscr{F}(Q)$. Dually, any isomorphism $\varphi:\mathscr{F}_1^{d-1}(P) \to \mathscr{F}_1^{d-1}(Q)$ extends to an isomorphism $\mathscr{F}(P) \to \mathscr{F}(Q)$.
\end{theorem}

\begin{theorem}[Gr\"unbaum {\cite[pp.~234]{grunbaum03}}]
    \label{middle_interval}
    Let $P,Q$ be $d$-polytopes with $d \geq 5$. Any isomorphism $\varphi:\mathscr{F}_1^{d-2}(P) \to \mathscr{F}_1^{d-2}(Q)$ extends to an isomorphism $\mathscr{F}(P) \to \mathscr{F}(Q)$.
\end{theorem}

Recall that for each $d \geq 3$, there exist distinct $d$-polytopes with identical $(d-3)$-skeleta. Dually, there exist distinct $d$-polytopes with identical dual $(d-3)$-skeleta. Thus, for $d \geq 5$, Theorem \ref{middle_interval} answers Problem \ref{reconstruction_problem} in the $n=1$ case: we can reconstruct every $d$-polytope $P$ from $\mathscr{F}_a^b(P)$ if and only if $\{1,\ldots,d-2\} \subseteq \{a,\ldots,b\}$.

\subsection{CW complexes}
CW complexes are a generalization of the boundary complexes of polytopes. In our 4-polytope reconstruction argument in Section \ref{s:4-polytopes}, we will consider CW complexes formed from subsets of a polytope's vertices, edges, and ridges.

\begin{definition}
    For $d \geq 1$, a \emph{CW $(d-1)$-complex} is a topological space built recursively as follows.
    \begin{itemize}
        \item A CW 0-complex is a finite set of points.
        \item To construct a CW $(d-1)$-complex with $d>1$, we first construct a CW $(d-2)$-complex $Y$. We then attach closed $(d-1)$-balls $F_1,\ldots,F_n$ to $Y$ with gluing maps $\partial F_1, \ldots, \partial F_n \to Y$.
    \end{itemize}
    The \emph{faces} of a CW $(d-1)$-complex are the closed 0-, \ldots, $(d-1)$-balls used to construct it, and the empty set. As with polytopes, a \emph{$k$-face} is a face of dimension $k$, and a 0, 1, $(d-2)$-, or $(d-1)$-face is called a \emph{vertex}, \emph{edge}, \emph{ridge}, or \emph{facet}, respectively.
\end{definition}

\begin{definition}
    A CW complex is \emph{regular} if each gluing map in its construction is a homeomorphism. It is \emph{strongly regular} if, in addition, the intersection of any two faces is a (possibly empty) face.
\end{definition}

If $X$ is a strongly regular CW $(d-1)$-complex, we define the face poset $\mathscr{F}(X)$, the $(a,b)$-skeleton $\mathscr{F}_a^b(X)$, and so on, identically as for a $d$-polytope.

\begin{definition}
    Let $X$ be a strongly regular CW complex and $Y \subseteq \mathscr{F}(X)$. The \emph{geometric realization} of $Y$ is the set $|Y| := \bigcup_{G \in Y}G \subseteq X$, endowed with the subspace topology.
\end{definition}

\begin{definition}
    A regular CW complex is \emph{pure} if each of its faces is contained in at least one facet.
\end{definition}

\begin{definition}
    A \emph{$(d-1)$-pseudomanifold} is a pure, strongly regular CW $(d-1)$-complex $X$ such that every ridge of $X$ is contained in exactly two facets.
\end{definition}

\begin{definition}
    A \emph{CW $(d-1)$-manifold} is a $(d-1)$-pseudomanifold homeomorphic to a manifold. In particular, a \emph{CW $(d-1)$-sphere} is a $(d-1)$-pseudomanifold homeomorphic to $S^{d-1}$.
\end{definition}

Note that the boundary complex of any $d$-polytope is a CW $(d-1)$-sphere.

\begin{definition}
    Let $X$ be a CW pseudomanifold. A \emph{singularity} of $X$ is a nonempty face $G \subseteq X$ such that for an arbitrary point $g$ in the relative interior of $G$, no open neighborhood of $g$ in $X$ is homeomorphic to a ball.
\end{definition}

A CW manifold is therefore a pseudomanifold with no singularities.

\subsection{Simplicial complexes}
Simplicial complexes will be the setting for our reconstruction results in Section \ref{s:simplicial}.

\begin{definition}
    A \emph{simplicial complex} is a strongly regular CW complex in which every face is a simplex.
\end{definition}

Suppose $X$ is a simplicial complex on a vertex set $V$. Sometimes, we will not regard the faces of $X$ as topological balls, but rather as subsets of $V$. So if $G$ is a $k$-simplex in $X$ with vertices $v_1, \ldots, v_{k+1} \in V$, we may write $G = \{v_1, \ldots, v_{k+1}\}$. If $B$ is the $(k-1)$-face of $G$ disjoint from $v_1$, we may write $G = B \cup v_1$ or $B = G \backslash v_1$.

\begin{definition}
    Let $G$ be a face of some simplicial complex $X$. The \emph{link} of $G$, denoted $\operatorname{Lk}(G,X)$, is the subcomplex of $X$ with faces $\{A \backslash G : A \in \mathscr{F}(X),G \subseteq A\}$.
\end{definition}

\begin{definition}
    Let $X$ be a connected, simplicial $(d-1)$-pseudomanifold. Then $X$ is \emph{normal} if for all $k$-faces $G$ with $0 \leq k \leq d-3$, $\operatorname{Lk}(G,X)$ is connected.
\end{definition}

The final complexes we introduce are \emph{homology manifolds}, which lie between normal pseudomanifolds and simplicial manifolds in generality. Homology manifolds may be defined with coefficients in $\mathds{Z}$ or in any fixed field; for the homology manifolds in this paper, our field of choice is $\mathds{Z}_2$.

\begin{definition}
    A \emph{homology $(d-1)$-manifold} is a connected, simplicial $(d-1)$-pseudomanifold $X$ such that for each $0 \leq k \leq d-2$ and $k$-face $G \subseteq X$, $\operatorname{Lk}(G,X)$ has the homology of a $(d-k-2)$-sphere. That is,
    \[ H_i(\operatorname{Lk}(G,X),\mathds{Z}_2) =
    \begin{cases}
        \mathds{Z}_2, & i=0\\
        0, & i=1,\ldots,d-k-3\\
        \mathds{Z}_2, & i=d-k-2
    \end{cases}.
    \]
\end{definition}

We will use the following theorem of Dancis to prove our positive reconstruction results for homology manifolds and pseudomanifolds.

\begin{theorem}[Dancis \cite{dancis84}]
    \label{dancis}
    Let $X$ and $Y$ be simplicial homology $(d-1)$-manifolds. Then any isomorphism $\varphi:\mathscr{F}_0^{\lceil d/2 \rceil}(P) \to \mathscr{F}_0^{\lceil d/2 \rceil}(Q)$ extends to an isomorphism $\mathscr{F}(P) \to \mathscr{F}(Q)$. If $d$ is odd and $H_{(d-1)/2}(X,\mathds{Z}_2) = H_{(d-1)/2}(Y,\mathds{Z}_2) = 0$, then any isomorphism $\varphi:\mathscr{F}_0^{(d-1)/2}(P) \to \mathscr{F}_0^{(d-1)/2}(Q)$ likewise extends to an isomorphism $\mathscr{F}(P) \to \mathscr{F}(Q)$.
\end{theorem}

\section{Reconstructing a 4-polytope from edges and ridges}
\label{s:4-polytopes}
In this section, we will prove that the face poset of a 4-polytope is uniquely determined by its edge-ridge incidences. We begin by showing how to reconstruct the facets of a CW 3-sphere from its $(1,2)$-skeleton.

\begin{lemma}
    \label{facet}
    Let $X$ be a strongly regular CW 3-sphere and $\mathscr{E} \subseteq \mathscr{F}^1(X), \mathscr{R} \subseteq \mathscr{F}^2(X)$ nonempty sets. Then there exists a facet $F$ with $\mathscr{F}_1^2(F) = \mathscr{E} \cup \mathscr{R}$ if and only if the following hold:
    \begin{enumerate}
        \item For each ridge $R \in \mathscr{R}$, every edge of $R$ is in $\mathscr{E}$.
        \item Each edge $E \in \mathscr{E}$ belongs to exactly two ridges in $\mathscr{R}$.
        \item The incidence graphs of $\mathscr{E} \cup \mathscr{R}$ and $\mathscr{F}_1^2(X)\backslash(\mathscr{E} \cup \mathscr{R})$ are connected.
    \end{enumerate}
\end{lemma}

\begin{proof}
    One direction is immediate: if there is such a facet, then $\mathscr{E}$ and $\mathscr{R}$ must satisfy conditions 1-3. To prove the other direction, assume $\mathscr{E}$ and $\mathscr{R}$ satisfy these conditions.

    Let $M \subset \mathscr{F}(X)$ be the set of all ridges in $\mathscr{R}$ and all their faces; that is, $M := \bigcup_{R \in \mathscr{R}} \mathscr{F}(R)$. Observe that $|M|$ is a connected 2-pseudomanifold embedded in $S^3$. We will show that the complement of $|M|$ in $S^3$ has two connected components.

    Since $|M|$ is a 2-pseudomanifold, any singularity of $|M|$ is located at a vertex. Let us construct a new manifold from $|M|$ embedded in $S^3$ by removing these singularities.

    Let $v$ be a vertex of $X$ which is a singularity in $|M|$. Draw an open ball $B \subset X$ around $v$, sufficiently small that $B$ only intersects the faces of $X$ containing $v$. Then $|M| \cap \partial B$ is a disjoint union of circles $C_1, \ldots, C_n$. Delete $|M| \cap B$ from $|M|$, and glue a disc $D_i$ to each circle $C_i$ so that $D_1, \ldots, D_n$ are disjoint and embedded in $B$. Repeat this operation for each singularity of $|M|$, and call the resulting manifold $M'$.

    Since $M'$ is a 2-manifold embedded in $S^3$, $M'$ must be orientable. Thus, $|M|$ is orientable as well. It follows that $\tilde{H}^2(|M|,\mathds{Z}) = \mathds{Z}$, so by Alexander duality, $\tilde{H}_0(X \backslash |M|,\mathds{Z}) = \mathds{Z}$. We may conclude that $X \backslash |M|$ has two connected components.

    Let $A,B$ be the connected components of $X \backslash |M|$. Then the closures of $A$ and $B$ in $X$ must each be the union of a nonempty set of facets. We claim that either $A$ or $B$ is the interior of a single facet.

    By way of contradiction, suppose $A$ and $B$ each contain the interiors of at least two facets. Let $F$ be a facet with $\operatorname{int} F \subset A$. Since $\operatorname{int} F \neq A$, we know $\partial F \not \subseteq M$, so there exists a ridge $R \in \partial F \backslash M$. Thus, $\operatorname{relint}R \subset A$. By a similar argument, there exists a ridge $R'$ such that $\operatorname{relint}R' \subset B$. However, this places $R$ and $R'$ in different connected components of the incidence graph of $\mathscr{F}_1^2(X) \backslash (\mathscr{E} \cup \mathscr{R})$, contradicting our connectedness assumption.

    We have shown that either $A$ or $B$ is the interior of some facet $F$. Thus, $\partial F \subseteq M$. Any edge $E$ in $\partial F$ belongs to two ridges in $\partial F$, and by assumption, these must be the only ridges in $M$ containing $E$. In other words, there can be no incidence between a face in $(\mathscr{E} \cup \mathscr{R}) \cap \partial F$ and a face in $(\mathscr{E} \cup \mathscr{R}) \backslash \partial F$. Since the incidence graph of $\mathscr{E} \cup \mathscr{R}$ is connected, it follows that $\partial F = M$ and $\mathscr{F}_1^2(F) = \mathscr{E} \cup \mathscr{R}$.
\end{proof}

\begin{lemma}
    \label{3-sphere}
    Let $X,Y$ be strongly regular CW 3-spheres and $\varphi:\mathscr{F}_1^2(X) \to \mathscr{F}_1^2(Y)$ an isomorphism. Then $\varphi$ extends to an isomorphism $\mathscr{F}_1^3(X) \to \mathscr{F}_1^3(Y)$.
\end{lemma}

\begin{proof}
    Let $\mathscr{E} \subseteq \mathscr{F}^1(X)$, $\mathscr{R} \subseteq \mathscr{F}^2(X)$. We can see that $\mathscr{E}, \mathscr{R}$ satisfy the conditions of Lemma \ref{facet} if and only if $\varphi(\mathscr{E}) \subseteq \mathscr{F}^1(Y), \varphi(\mathscr{R}) \subseteq \mathscr{F}^2(Y)$ satisfy those conditions. Thus, for each facet $F$ of $X$, we may define $\varphi(F):=F'$, where $F'$ is the unique facet of $Y$ such that
    \[
        \varphi\left(\mathscr{F}_1^2(F)\right) = \mathscr{F}_1^2(F').
    \]
    This gives us the desired isomorphism $\varphi:\mathscr{F}_1^3(X) \to \mathscr{F}_1^3(Y)$.
\end{proof}

Our theorem follows from Lemma \ref{3-sphere} by duality.

\begin{theorem}
    \label{edge_ridge}
    Let $P, Q$ be 4-polytopes and $\varphi:\mathscr{F}_1^2(P) \to \mathscr{F}_1^2(Q)$ an isomorphism. Then $\varphi$ extends to an isomorphism $\mathscr{F}(P) \to \mathscr{F}(Q)$.
\end{theorem}

\section{The \texorpdfstring{$(d-3)$}{(d-3)}-skeleton and dual \texorpdfstring{$(d-3)$}{(d-3)}-skeleton are not enough}
\label{s:wedges}

In this section, we will show that a $d$-polytope cannot generally be identified from the combined information of its $(d-3)$-skeleton and dual $(d-3)$-skeleton. For each $d \geq 3$, we will construct combinatorially distinct $d$-polytopes $P,Q$ whose $(d-3)$-skeleta are isomorphic and whose dual $(d-3)$-skeleta are also isomorphic.

Fix $d \geq 3$, and let $e_1, \ldots, e_d$ be the standard basis vectors of $\mathds{R}^d$. Consider the $d$-polytope
\[
    W := \operatorname{conv} \left\{ \pm e_1, \ldots, \pm e_{d-1}, 2e_d - e_1, e_d \pm e_2, \ldots, e_d \pm e_{d-1} \right\}.
\]
Here, $W$ is a wedge of two $(d-1)$-dimensional cross-polytopes $F$ and $F'$ sharing a single vertex at $e_1$. Specifically, $F$ is the $(d-1)$-dimensional cross-polytope with vertices $\{\pm e_1, \ldots, \pm e_{d-1}\}$, and $F'$ is the $(d-1)$-dimensional cross-polytope with vertices $\{e_1, 2e_d - e_1, e_d \pm e_2, \ldots, e_d \pm e_{d-1}\}$. The proper, nonempty faces of $W$ are as follows:
\begin{itemize}
    \item facets $F$ and $F'$,
    \item the vertex at $e_1$,
    \item for each $k$-face $G$ of $F$, $0 \leq k \leq d-2$, excluding the vertex at $e_1$:
    \begin{itemize}
        \item $G$ itself,
        \item a corresponding $k$-face $G'$ of $F'$, directly ``above" $G$ in the $e_d$ direction,
        \item a ``lateral" $(k+1)$-face equal to the convex hull of $G$ and $G'$.
    \end{itemize}
\end{itemize}

We construct new $d$-polytopes $P$ and $Q$ from $W$, each by shifting a vertex of $F'$ in the $-e_d$ direction. To construct $P$, we move the vertex at $2e_d - e_1$ to $e_d - e_1$. To construct $Q$, we leave the vertex at $2e_d - e_1$ put, and instead move the vertex at $e_d + e_2$ to $\frac{1}{2}e_d + e_2$. So
\begin{align*}
    P &:= \operatorname{conv} \{\pm e_1, \ldots, \pm e_{d-1}, e_d - e_1, e_d \pm e_2, \ldots, e_d \pm e_{d-1} \},\\
    Q &:= \operatorname{conv} \left\{\pm e_1, \ldots, \pm e_{d-1}, 2e_d - e_1, \frac{1}{2} e_d + e_2, e_d - e_2, e_d \pm e_3, \ldots, e_d \pm e_{d-1} \right\}.
\end{align*}

When we construct $P$ from $W$, we introduce a new ridge with vertices $\{e_d \pm e_2, \ldots, e_d \pm e_{d-1}\}$, subdividing $F'$ into two pyramidal facets. One of these pyramidal facets contains the vertex at $e_1$; the other contains the vertex at $e_d - e_1$. All faces of $W$ besides $F'$ remain intact.

Similarly, when we construct $Q$ from $W$, we introduce a new ridge with vertices $\{e_1, 2e_d - e_1, e_d \pm e_3, \ldots, e_d \pm e_{d-1}\}$, again subdividing $F'$ into two pyramidal facets. One of these pyramidal facets contains the vertex at $\frac{1}{2}e_d + e_2$; the other contains the vertex at $e_d - e_2$. As before, all faces of $W$ besides $F'$ remain intact. See Figures \ref{fig0}-\ref{fig1} for a comparison of $P$ and $Q$ when $d=3,4$.

\begin{claim}
    $\mathscr{F}(P)$ is not isomorphic to $\mathscr{F}(Q)$.
\end{claim}

\begin{proof}
    Consider the vertex of $Q$ at $e_1$. This vertex is contained in $2^{d-2}$ facets of $F$, meaning $e_1$ is contained in $2^{d-2}$ lateral facets of $Q$. Additionally, $e_1$ is contained in $F$ itself, as well as the pyramidal facets $\operatorname{conv}\{e_1, 2e_d - e_1, \frac{1}{2}e_d + e_2, e_d \pm e_3, \ldots, e_d \pm e_{d-1}\}$ and $\operatorname{conv}\{e_1, 2e_d - e_1, e_d - e_2, e_d \pm e_3, \ldots, e_d \pm e_{d-1}\}$. Thus, $e_1$ belongs to a total of $2^{d-2} + 3$ facets of $Q$.

    We will show that no vertex of $P$ belongs to this many facets. Let $v$ be a vertex of $P$ contained in $F$. Then $v$ belongs to $2^{d-2}$ facets of $F$, so $v$ is contained in $2^{d-2}$ lateral facets of $P$. Furthermore, if $v \neq e_1$, then the vertex $v + e_d$ belongs to the same $2^{d-2}$ lateral facets of $P$. Thus, each vertex of $P$ is contained in exactly $2^{d-2}$ lateral facets.

    There are only three non-lateral facets of $P$: the cross-polytope $F$ and the pyramids $\operatorname{conv}\{e_1, e_d \pm e_2, \ldots, e_d \pm e_{d-1}\}$ and $\operatorname{conv}\{e_d - e_1, e_d \pm e_2, \ldots, e_d \pm e_{d-1}\}$. There is no vertex shared by all three non-lateral facets, so each vertex of $P$ belongs to at most $2^{d-2} + 2$ facets in total.

    Since $Q$ has a vertex contained in $2^{d-2} + 3$ facets, and $P$ does not, we may conclude that $P$ and $Q$ are combinatorially distinct.
\end{proof}

\begin{claim}
    $\mathscr{F}_0^{d-3}(P)$ is isomorphic to $\mathscr{F}_0^{d-3}(Q)$.
\end{claim}

\begin{proof}
    Recall that when we constructed each of $P,Q$ from $W$, we introduced one new ridge and subdivided the facet $F'$ into two pyramidal facets, but we preserved all faces of dimension $d-3$ or lower. Thus, $\mathscr{F}_0^{d-3}(P) \cong \mathscr{F}_0^{d-3}(Q) \cong \mathscr{F}_0^{d-3}(W)$.
\end{proof}

\begin{claim}
    $\mathscr{F}_2^{d-1}(P)$ is isomorphic to $\mathscr{F}_2^{d-1}(Q)$.
\end{claim}

\begin{proof}
    Let $A,A'$ be the two facets of $P$ resulting from our subdivision of $F'$:
    \begin{align*}
        A &:= \operatorname{conv}\{e_1, e_d \pm e_2, \ldots, e_d \pm e_{d-1}\},\\
        A' &:= \operatorname{conv}\{e_d - e_1, e_d \pm e_2, \ldots, e_d \pm e_{d-1}\}.
    \end{align*}
    Let $B,B'$ be the analogous facets of $Q$:
    \begin{align*}
        B &:= \operatorname{conv}\left\{e_1, 2e_d - e_1, \frac{1}{2}e_d + e_2, e_d \pm e_3, \ldots, e_d \pm e_{d-1}\right\},\\
        B' &:= \operatorname{conv}\{e_1, 2e_d - e_1, e_d - e_2, e_d \pm e_3, \ldots, e_d \pm e_{d-1}\}.
    \end{align*}
    Then $A,A'$ are pyramids over a common $(d-2)$-cross-polytope, as are $B,B'$. Thus, there is an isomorphism $\psi:\mathscr{F}(A) \cup \mathscr{F}(A') \to \mathscr{F}(B) \cup \mathscr{F}(B')$ such that $\psi(\partial(A \cup A')) = \partial(B \cup B')$.

    We may construct $\mathscr{F}_2^{d-1}(P)$ from $\mathscr{F}(A) \cup \mathscr{F}(A')$ in the following steps:
    \begin{enumerate}
        \item For each $k$-face $G \in \partial(A \cup A')$ with $k \geq 1$, add a ``lateral" $(k+1)$-face $L(G)$ containing $G$. Let $G,L(G) \subseteq L(G')$ if and only if $G \subseteq G'$.
        \item For each $k$-face $G \in \partial(A \cup A')$ with $k \geq 2$, add an ``opposite" $k$-face $O(G)$ contained in $L(G)$ and disjoint from $A \cup A'$. Let $O(G) \subseteq O(G'),L(G')$ if and only if $G \subseteq G'$.
        \item Add one facet $F$ containing exactly the faces $O(G)$ added in step 2.
    \end{enumerate}
    We may construct $\mathscr{F}_2^{d-1}(Q)$ from $\mathscr{F}(B) \cup \mathscr{F}(B')$ in the same steps, simply replacing $A \cup A'$ with $B \cup B'$ at each occurrence. It follows that $\mathscr{F}_2^{d-1}(P) \cong \mathscr{F}_2^{d-1}(Q)$.
\end{proof}

In conclusion,
\begin{theorem}
    \label{identical_skeleta}
    For all $d \geq 3$, there exist $d$-polytopes $P$ and $Q$ such that
    \begin{align*}
        \mathscr{F}_0^{d-3}(P) &\cong \mathscr{F}_0^{d-3}(Q),\\
        \mathscr{F}_2^{d-1}(P) &\cong \mathscr{F}_2^{d-1}(Q),\\
        \mathscr{F}(P) &\not\cong \mathscr{F}(Q).
    \end{align*}
\end{theorem}

Note that when $d=3$, polytopes satisfying Theorem \ref{identical_skeleta} are unremarkable. It simply means that $f_0(P) = f_0(Q)$ and $f_2(P) = f_2(Q)$, a property shared by many pairs of distinct 3-polytopes. Theorem \ref{identical_skeleta} is nontrivial for $d \geq 4$.

\begin{remark}
    When $d=4$, Theorem \ref{identical_skeleta} gives combinatorially distinct 4-polytopes $P$ and $Q$ with the same graph and the same dual graph.
\end{remark}

\begin{figure}
    \centering
    \tdplotsetmaincoords{65}{75}
    \begin{tikzpicture}[line cap=round, line join=round, tdplot_main_coords, scale=2.5, line width = 1pt, label distance=-4pt]
        \coordinate (e1) at (1,0,0);
        \coordinate (-e1) at (-1,0,0);
        \coordinate (e2) at (0,1,0);
        \coordinate (-e2) at (0,-1,0);
        \coordinate (-e1+) at (-1,0,2);
        \coordinate (e2+) at (0,1,1);
        \coordinate (-e2+) at (0,-1,1);

        \draw (e2) -- (e1) -- (-e2);
        \draw [dashed](e2) -- (-e1) -- (-e2);
        \draw (e1) -- (e2+) -- (-e1+) -- (-e2+) -- cycle;
        \draw [dashed] (-e1) -- (-e1+);
        \draw (e2) -- (e2+);
        \draw (-e2) -- (-e2+);
    \end{tikzpicture}
    \qquad
    \begin{tikzpicture}[line cap=round, line join=round, tdplot_main_coords, scale=2.5, line width = 1pt, label distance=-4pt]
        \coordinate (e1) at (1,0,0);
        \coordinate (-e1) at (-1,0,0);
        \coordinate (e2) at (0,1,0);
        \coordinate (-e2) at (0,-1,0);
        \coordinate (-e1+) at (-1,0,1);
        \coordinate (e2+) at (0,1,1);
        \coordinate (-e2+) at (0,-1,1);

        \draw (e2) -- (e1) -- (-e2);
        \draw [dashed] (e2) -- (-e1) -- (-e2);
        \draw (e1) -- (e2+) -- (-e1+) -- (-e2+) -- cycle;
        \draw [dashed] (-e1) -- (-e1+);
        \draw (e2) -- (e2+);
        \draw (-e2) -- (-e2+);

        \draw [line width=2pt, shorten <=1pt, shorten >=1pt] (e2+) -- (-e2+);
    \end{tikzpicture}
    \qquad
    \begin{tikzpicture}[line cap=round, line join=round, tdplot_main_coords, scale=2.5, line width = 1pt, label distance=-4pt]
        \coordinate (e1) at (1,0,0);
        \coordinate (-e1) at (-1,0,0);
        \coordinate (e2) at (0,1,0);
        \coordinate (-e2) at (0,-1,0);
        \coordinate (-e1+) at (-1,0,2);
        \coordinate (e2+) at (0,1,0.5);
        \coordinate (-e2+) at (0,-1,1);

        \draw (e2) -- (e1) -- (-e2);
        \draw [dashed] (e2) -- (-e1) -- (-e2);
        \draw (e1) -- (e2+) -- (-e1+) -- (-e2+) -- cycle;
        \draw [dashed] (-e1) -- (-e1+);
        \draw (e2) -- (e2+);
        \draw (-e2) -- (-e2+);

        \draw [line width=2pt, shorten <=1pt, shorten >=1pt] (e1) -- (-e1+);
    \end{tikzpicture}

    \caption{From left to right, polytopes $W$, $P$, and $Q$ for $d=3$. The square base is $F$; the bold edge is introduced when we construct $P$ or $Q$ from $W$.}

    \label{fig0}
\end{figure}
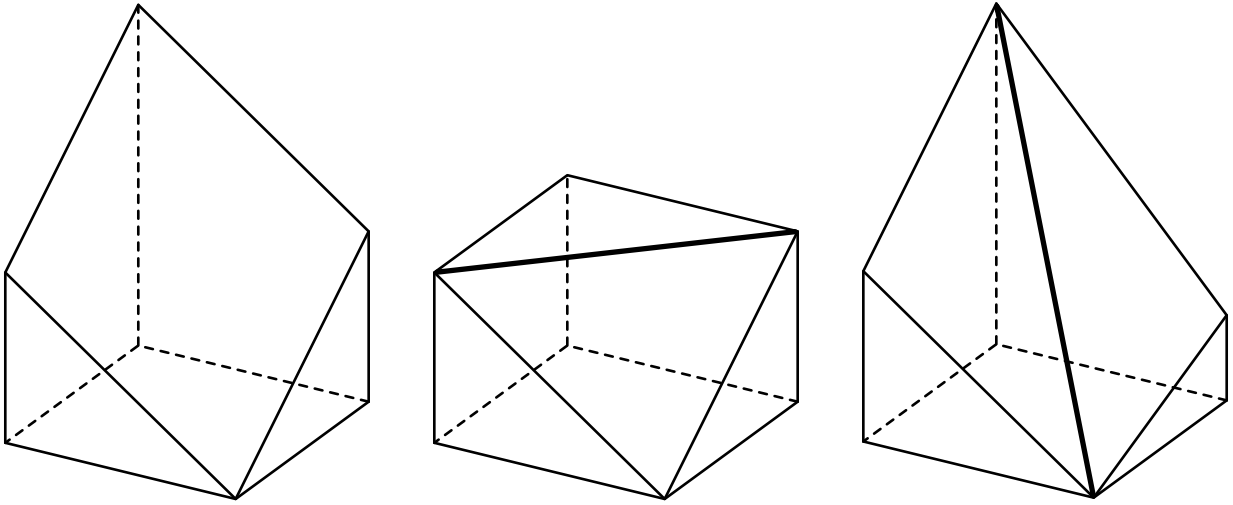

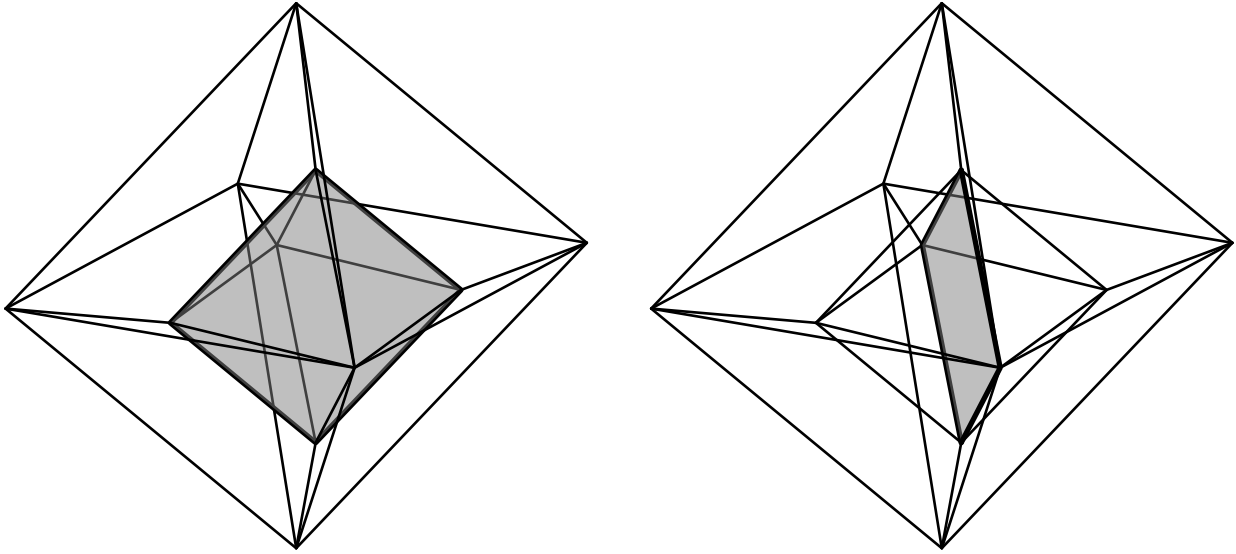
\begin{figure}
    \centering
    \tdplotsetmaincoords{65}{75}
    \begin{tikzpicture}[line cap=round, line join=round, tdplot_main_coords, scale=2.0, line width = 1pt, label distance=-4pt]
        \coordinate (e1) at (1, 0, 0);
        \coordinate (-e1) at (-1, 0, 0);
        \coordinate (e2) at (0, 1, 0);
        \coordinate (-e2) at (0, -1, 0);
        \coordinate (e3) at (0, 0, 1);
        \coordinate (-e3) at (0, 0, -1);

        \coordinate (-e1+) at (-2, 0, 0);
        \coordinate (e2+) at (-0.5, 2, 0);
        \coordinate (-e2+) at (-0.5, -2, 0);
        \coordinate (e3+) at (-0.5, 0, 2);
        \coordinate (-e3+) at (-0.5, 0, -2);

        \draw (e2) -- (-e1) -- (-e2);
        \draw (e3) -- (-e1) -- (-e3);
        \draw [line width=2pt] (e2) -- (e3) -- (-e2) -- (-e3) -- cycle;

        \draw (e2+) -- (-e1+) -- (-e2+);
        \draw (e3+) -- (-e1+) -- (-e3+);
        \draw (e2+) -- (e3+) -- (-e2+) -- (-e3+) -- cycle;

        \draw (-e1) -- (-e1+);
        \draw (e2) -- (e2+);
        \draw (-e2) -- (-e2+);
        \draw (e3) -- (e3+);
        \draw (-e3) -- (-e3+);

        \path[fill=gray, opacity=0.5] (e2) -- (e3) -- (-e2) -- (-e3) -- cycle;

        \draw (e2) -- (e1) -- (-e2);
        \draw (e3) -- (e1) -- (-e3);
        \draw (e2+) -- (e1) -- (-e2+);
        \draw (e3+) -- (e1) -- (-e3+);
    \end{tikzpicture}
    \qquad
    \begin{tikzpicture}[line cap=round, line join=round, tdplot_main_coords, scale=2.0, line width = 1pt, label distance=-4pt]
        \coordinate (e1) at (1, 0, 0);
        \coordinate (-e1) at (-1, 0, 0);
        \coordinate (e2) at (0, 1, 0);
        \coordinate (-e2) at (0, -1, 0);
        \coordinate (e3) at (0, 0, 1);
        \coordinate (-e3) at (0, 0, -1);

        \coordinate (-e1+) at (-2, 0, 0);
        \coordinate (e2+) at (-0.5, 2, 0);
        \coordinate (-e2+) at (-0.5, -2, 0);
        \coordinate (e3+) at (-0.5, 0, 2);
        \coordinate (-e3+) at (-0.5, 0, -2);

        \draw (e2) -- (-e1) -- (-e2);
        \draw [line width=2pt] (e3) -- (-e1) -- (-e3);
        \draw (e2) -- (e3) -- (-e2) -- (-e3) -- cycle;

        \draw (e2+) -- (-e1+) -- (-e2+);
        \draw (e3+) -- (-e1+) -- (-e3+);
        \draw (e2+) -- (e3+) -- (-e2+) -- (-e3+) -- cycle;

        \draw (-e1) -- (-e1+);
        \draw (e2) -- (e2+);
        \draw (-e2) -- (-e2+);
        \draw (e3) -- (e3+);
        \draw (-e3) -- (-e3+);

        \path[fill=gray, opacity=0.5] (e1) -- (e3) -- (-e1) -- (-e3) -- cycle;

        \draw (e2) -- (e1) -- (-e2);
        \draw [line width=2pt] (e3) -- (e1) -- (-e3);
        \draw (e2+) -- (e1) -- (-e2+);
        \draw (e3+) -- (e1) -- (-e3+);
    \end{tikzpicture}

    \caption{Schlegel diagrams of $P$ (left) and $Q$ (right) for $d=4$. The outer octahedron is $F$; the shaded square is the ridge introduced when we construct $P$ or $Q$ from $W$ (not pictured).}

    \label{fig1}
\end{figure}

\section{The \texorpdfstring{$(1,d-2)$}{(1,d-2)}-skeleton is necessary and sufficient}
\label{s:main-result}

We are now ready to answer Problem \ref{reconstruction_problem} for $d \geq 4$.

\begin{theorem}
    \label{reconstructing}
    Fix $d \geq 4$, and let
    \begin{gather*}
        0 \leq a_1 \leq b_1 < d,\\
        \vdots\\
        0 \leq a_n \leq b_n < d.
    \end{gather*}
    The following are equivalent:
    \begin{enumerate}[label=(\roman*)]
        \item For all $d$-polytopes $P$ and $Q$, if $\mathscr{F}_{a_i}^{b_i}(P) \cong \mathscr{F}_{a_i}^{b_i}(Q)$ for $i=1, \ldots, n$, then $\mathscr{F}(P) \cong \mathscr{F}(Q)$.
        \item $\{1, \ldots, d-2\} \subseteq \{a_i, \ldots, b_i\}$ for some $1 \leq i \leq n$.
    \end{enumerate}
\end{theorem}

\begin{proof}
    By Theorems \ref{middle_interval} and \ref{edge_ridge}, (ii) implies (i).

    Suppose (i) holds, and let $P,Q$ be $d$-polytopes as in Theorem \ref{identical_skeleta}. Then $\mathscr{F}(P) \not\cong \mathscr{F}(Q)$, so there exists $1 \leq i \leq n$ such that $\mathscr{F}_{a_i}^{b_i}(P) \not\cong \mathscr{F}_{a_i}^{b_i}(Q)$. We know $\mathscr{F}_0^{d-3}(P) \cong \mathscr{F}_0^{d-3}(Q)$ and $\mathscr{F}_2^{d-1}(P) \cong \mathscr{F}_2^{d-1}(Q)$, so
    \[
        \{a_i, \ldots, b_i\} \not \subseteq \{0, \ldots, d-3\}, \{2, \ldots, d-1\}.
    \]
    Thus, $\{1, \ldots, d-2\} \subseteq \{a_i, \ldots, b_i\}$. We may conclude that (i) implies (ii).
\end{proof}

If we wish to identify all $d$-polytopes from a set of $(a,b)$-skeleta, then knowing more than one skeleton does not help. Theorem \ref{reconstructing} shows this for $d \geq 4$: either we know the $(1,d-2)$-skeleton, which is sufficient on its own, or we do not have enough information. We can make analogous statements for $d=2,3$ (Observations \ref{d=2}-\ref{d=3}).

Of course, there are \emph{specific} $d$-polytopes which can be identified from a set of a $(a,b)$-skeleta, but not from any one of those skeleta alone. For the simplest example, let $d=3$ and $P$ be the square pyramid. Then we can identify $P$ from $\mathscr{F}^0(P)$ and $\mathscr{F}^2(P)$, as $P$ is the only 3-polytope with five vertices and five facets. However, we cannot identify $P$ from $\mathscr{F}^0(P)$ or $\mathscr{F}^2(P)$ alone, as $P$ is neither the only 3-polytope with five vertices nor the only 3-polytope with five facets.

We conclude this section with a still-open problem of Gr\"unbaum {\cite[pp.~234]{grunbaum03}}.

\begin{problem}[Gr\"unbaum]
    For $d \geq 5$, do there exist $d$-polytopes $P,Q$ with an isomorphism $\varphi:\mathscr{F}_{d-3}^{d-2}(P) \to \mathscr{F}_{d-3}^{d-2}(Q)$ that does not extend to an isomorphism $\mathscr{F}_{d-3}^{d-1}(P) \to \mathscr{F}_{d-3}^{d-1}(Q)$?
\end{problem}

\section{Simplicial complexes}
Our focus so far has been on reconstructing polytopes, but we can ask the same questions about any class of objects with graded face posets: CW spheres, CW manifolds, simplicial manifolds, pseudomanifolds, and so on. In this section, we prove positive and negative reconstruction results for homology manifolds and normal, simplicial pseudomanifolds.

Recall that when we discuss homology manifolds, the homologies in question have $\mathds{Z}_2$-coefficients.

\label{s:simplicial}
\begin{lemma}
\label{simplicial_one_down}
    Let $X,Y$ be normal, simplicial $(d-1)$-pseudomanifolds with $d \geq 4$, and let $2 \leq k \leq d-2$. Then any isomorphism $\varphi:\mathscr{F}_{k-1}^k(X) \to \mathscr{F}_{k-1}^k(Y)$ extends to an isomorphism $\mathscr{F}_{k-2}^k(X) \to \mathscr{F}_{k-2}^k(Y)$.
\end{lemma}

\begin{proof}
    We will call a subset $T$ of $\mathscr{F}_{k-1}^k(X)$ \emph{triangular} if $T = \{A_1, A_2, A_3, B_1, B_2, B_3\}$ for some distinct $k$-faces $A_1, A_2, A_3$ and distinct $(k-1)$-faces $B_1, B_2, B_3$, and
    \begin{align*}
        B_1 &\subset A_2, A_3,\\
        B_2 &\subset A_1, A_3,\\
        B_3 &\subset A_1, A_2.
    \end{align*} 
    Let $\sim$ be the finest equivalence relation on triangular subsets of $\mathscr{F}_{k-1}^k(X)$ with the following property: if $T, T'$ are triangular subsets sharing two $(k-1)$-faces, then $T \sim T'$. For any triangular subset $T$, we denote by $[T]$ the equivalence class of $T$ under $\sim$.

    For all triangular subsets $T = \{A_1, A_2, A_3, B_1, B_2, B_3\}$ as above, define
    \[
        I(T) := A_1 \cap A_2 \cap A_3.
    \]
    Since $A_1, A_2, A_3$ are distinct, we can see that $A_1 \cap A_2 = B_3$ and $A_1 \cap A_3 = B_2$. Thus, $I(T) = B_2 \cap B_3$. Since $B_2$ and $B_3$ are $(k-1)$ simplices belonging to the $k$-simplex $A_1$, it follows that $I(T)$ is a $(k-2)$-face of $X$. Furthermore, if $B_2, B_3 \in T'$ for some other triangular subset $T'$, then $I(T') = B_2 \cap B_3 = I(T)$. As a result, for all triangular subsets $T, T'$ with $T \sim T'$, we have $I(T) = I(T')$. We may therefore define $I[T] := I(T)$ as a function on equivalence classes.
    
    Let $C$ be a $(k-2)$-face of $X$. We will prove that $C = I[T]$ for exactly one equivalence class $[T]$. Since $k \leq d-2$, we know $C \subseteq G$ for some $(k+1)$-face $G$ of $X$, and the closed interval $[C,G]$ in $\mathscr{F}(X)$ is a Boolean lattice of length three. Thus, the open interval $(C,G)$ in $\mathscr{F}(X)$ is a triangular subset, and $C = I[(C,G)]$.

    We have shown that $C = I[T]$ for some triangular subset $T$; it remains to prove that $[T]$ is unique. Suppose $C = I[T] = I[T']$ for some triangular subsets $T, T'$. Let $A \in T$ be a $k$-face and let $B_1, B_2 \in T$ be $(k-1)$-faces with $B_1, B_2 \subset A$. Let $G$ be a $(k+1)$-face of $X$ with $A \subset G$. As before, the open interval $(C,G)$ in $\mathscr{F}(X)$ is a triangular poset. Furthermore, $B_1, B_2 \in (C,G)$, so $(C,G) \sim T$. By the same reasoning, we can find a $(k+1)$-face $G'$ of $X$ such that $(C,G') \sim T'$.

    Since $X$ is a normal pseudomanifold, we know that $\operatorname{Lk}(C,X)$ is a normal pseudomanifold of dimension at least 2, and $G \backslash C, G' \backslash C$ are 2-faces of $\operatorname{Lk}(C,X)$. Thus, there is a sequence of 2-faces $G\backslash C = U_1, \ldots, U_n = G'\backslash C \subseteq \operatorname{Lk}(C,X)$ such that $U_m, U_{m+1}$ share an edge for each $1 \leq m < n$. For $1 \leq m \leq n$, let $U_m = G_m \backslash C$ with $G_m$ a $(k+1)$-face of $X$, so $G_1=G$ and $G_n = G'$. Then $(C,G_1), \ldots, (C,G_n)$ are triangular subsets, and
    \[
        (C,G) = (C,G_1) \sim \cdots \sim (C,G_n) = (C,G').
    \]
    Thus, $T \sim T'$. It follows that $C=I[T]$ for a unique equivalence class $[T]$ of triangular subsets.

    Define triangular subsets of $\mathscr{F}_{k-1}^k(Y)$, equivalence classes, and the function $I$ identically as for $X$. Then the isomorphism $\varphi:\mathscr{F}_{k-1}^k(X) \to \mathscr{F}_{k-1}^k(Y)$ takes triangular subsets to triangular subsets and preserves equivalence classes. Thus, we can extend $\varphi$ to an isomorphism $\mathscr{F}_{k-2}^k(X) \to \mathscr{F}_{k-2}^k(Y)$
    by letting
    \[
        \varphi(I[T]) := I[\varphi(T)]
    \]
    for all triangular subsets $T$ of $\mathscr{F}_{k-1}^k(X)$.    
\end{proof}

By applying Lemma \ref{simplicial_one_down} repeatedly, we can use the $(k-1,k)$-skeleton of a normal, simplicial pseudomanifold to reconstruct its entire $k$-skeleton.

\begin{lemma}
\label{simplicial_reconstruct_down}
    Let $X,Y$ be normal, simplicial $(d-1)$-pseudomanifolds with $d \geq 3$, and let $1 \leq k \leq d-2$. Then any isomorphism $\varphi:\mathscr{F}_{k-1}^k(X) \to \mathscr{F}_{k-1}^k(Y)$ extends to an isomorphism $\mathscr{F}_0^k(X) \to \mathscr{F}_0^k(Y)$.
\end{lemma}

Our theorem on homology manifolds now follows from Theorem \ref{dancis} and Lemma \ref{simplicial_reconstruct_down}.

\begin{theorem}
\label{half_dimension}
    Let $X,Y$ be simplicial homology $(d-1)$-manifolds with $d \geq 4$, and let $\lceil\frac{d}{2}\rceil \leq k \leq d-2$. Then any isomorphism $\varphi:\mathscr{F}_{k-1}^{k}(X) \to \mathscr{F}_{k-1}^{k}(Y)$ extends to an isomorphism $\mathscr{F}(X) \to \mathscr{F}(Y)$. If $d$ is odd and $H_{(d-1)/2}(X,\mathds{Z}_2)=H_{(d-1)/2}(Y,\mathds{Z}_2)=0$, then the same holds for $d \geq 3$ and $\frac{d-1}{2} \leq k \leq d-2$.
\end{theorem}

\begin{proof}
    By Lemma \ref{simplicial_reconstruct_down}, we may extend $\varphi$ to an isomorphism $\mathscr{F}_0^k(X) \to \mathscr{F}_0^k(Y)$. By Theorem \ref{dancis}, we may then extend $\varphi$ to an isomorphism $\mathscr{F}(X) \to \mathscr{F}(Y)$.
\end{proof}

The situation changes for pseudomanifolds, as Theorem \ref{dancis} no longer applies directly. Nevertheless, we can generalize Theorem \ref{half_dimension} to normal pseudomanifolds whose face links meet certain criteria.

\begin{theorem}
\label{reconstruct_pseudomanifold}
    Let $d \geq 5$ and $\lceil \frac{d+1}{2} \rceil \leq k \leq d-2$. Let $X$ and $Y$ be normal, simplicial $(d-1)$-pseudomanifolds in which the link of each $(2k-d-1)$-face is a homology $(2d-2k-1)$-manifold. Then any isomorphism $\varphi:\mathscr{F}_{k-1}^k(X) \to \mathscr{F}_{k-1}^k(Y)$ extends to an isomorphism $\mathscr{F}(X) \to \mathscr{F}(Y)$.
\end{theorem}

\begin{proof}
    By Lemma \ref{simplicial_reconstruct_down}, we can extend $\varphi$ to an isomorphism $\mathscr{F}_0^k(X) \to \mathscr{F}_0^k(Y)$. It remains to prove that we can extend the isomorphism to higher-dimensional faces.

    Let $G$ be a $(2k-d-1)$-face of $X$. Then $\varphi|_{\mathscr{F}_0^{d-k}(\operatorname{Lk}(G,X))}:\mathscr{F}_0^{d-k}(\operatorname{Lk}(G,X)) \to \mathscr{F}_0^{d-k}(\operatorname{Lk}(\varphi(G),Y))$ is an isomorphism. By assumption, $\operatorname{Lk}(G,X)$ and $\operatorname{Lk}(\varphi(G),Y)$ are homology $(2d-2k-1)$-manifolds. Thus, by Theorem \ref{dancis}, $\varphi|_{\mathscr{F}_0^{d-k}(\operatorname{Lk}(G,X))}$ extends to an isomorphism $\mathscr{F}(\operatorname{Lk}(G,X)) \to \mathscr{F}(\operatorname{Lk}(\varphi(G),Y))$.

    Let $k < r < d$, and let $v_1, \ldots, v_{r+1}$ be vertices in $X$. Then the following are equivalent:
    \begin{itemize}
        \item $\{v_1, \ldots, v_{r+1}\}$ is the vertex set of an $r$-face in $X$.
        \item $\{v_1, \ldots, v_{2k-d}\}$ is the vertex set of a face $G \in \mathscr{F}^{2k-d-1}(X)$, and $\{v_{2k-d+1},\ldots,v_{r+1}\}$ is the vertex set of a $(d-2k+r)$-simplex in $\operatorname{Lk}(G,X)$.
        \item $\varphi\{v_1, \ldots, v_{2k-d}\}$ is the vertex set of a face $\varphi(G) \in \mathscr{F}^{2k-d-1}(Y)$, and $\varphi\{v_{2k-d+1},\ldots,v_{r+1}\}$ is the vertex set of a $(d-2k+r)$-simplex in $\operatorname{Lk}(\varphi(G),Y)$.
        \item $\varphi\{v_1, \ldots, v_{r+1}\}$ is the vertex set of an $r$-face in $Y$.
    \end{itemize}
    
    For each $r$-face $G$ in $X$ with $k < r < d$, we may now define $\varphi(G):=G'$, where $G'$ is the unique $r$-face of $Y$ such that $\varphi(\mathscr{F}^0(G)) = \mathscr{F}^0(G')$. This gives our desired isomorphism $\varphi:\mathscr{F}(X) \to \mathscr{F}(Y)$.
\end{proof}

\begin{remark}
    Theorem \ref{reconstruct_pseudomanifold} holds for $d \geq 4$ and $\lceil \frac{d}{2} \rceil \leq k \leq d-2$ if we instead require the link of each $(2k-d)$-face to be a homology $(2d-2k-2)$-manifold having trivial $(d-k-1)$-homology with $\mathds{Z}_2$-coefficients.
\end{remark}

Let us illustrate why the conditions of Theorem \ref{reconstruct_pseudomanifold} are important. For arbitrary $d \geq 3$, we will construct a pair of normal, simplicial $(d-1)$-pseudomanifolds $X$ and $Y$ which are not isomorphic, but whose $(d-2)$-skeleta \emph{are} isomorphic.

For each $d \geq 2$, let
\[
    V_{d-1} := \{1,\ldots,d\} \times \{0,1,2\}.
\]
We define $M^{d-1}$ as the pure $(d-1)$-complex on $V_{d-1}$ whose facets are exactly the sets $\{(1,a_1),\ldots,(d,a_d)\} \subset V_{d-1}$ with
\[
    \sum_{i=1}^d a_i \not\equiv 0\pmod 3.
\]
For $0 \leq k \leq d-2$, the $k$-faces of $M^{d-1}$ are the sets $\{(m_1,a_1),\ldots,(m_{k+1},a_{k+1})\} \subset V$ with $m_1 < \ldots < m_{k+1}$. To see why, observe that regardless of $a_1, \ldots, a_{k+1}$, we can always choose $a_{k+2}, \ldots, a_d \in \{0,1,2\}$ such that $\sum_{i=1}^d a_i \not\equiv 0\pmod 3$.

Each $(d-2)$-face
\[
    G = \{(1,a_1),\ldots,(n-1,a_{n-1}),(n+1,a_{n+1}),\ldots,(d,a_d)\}
\]
of $M^{d-1}$ is contained in exactly two facets. These facets are $G \cup \{(n,a_n)\}$ and $G \cup \{(n,a_n')\}$, where
\begin{align*}
    a_n &\equiv 1 - \sum_{i=1}^{n-1}a_i - \sum_{i=n+1}^d a_i \pmod 3,\\
    a_n' &\equiv 2 - \sum_{i=1}^{n-1}a_i - \sum_{i=n+1}^d a_i \pmod 3.
\end{align*}
Thus, $M^{d-1}$ is a pseudomanifold for each $d \geq 2$.

Observe that $M^1$ is a six-vertex circle and therefore a 1-manifold. Additionally, for each $d \geq 3$ and $G \in \mathscr{F}^k(M^{d-1})$ with $0 \leq k \leq d-3$, we have $\operatorname{Lk}(G,M^{d-1}) \cong M^{d-k-2}$. By strong induction on $d$, we can see that $M^{d-1}$ is a normal $(d-1)$-pseudomanifold for all $d \geq 2$.

Furthermore, $M^{d-1}$ is orientable. For each facet $\{(1,a_1),\ldots,(d,a_d)\}$, we may set $[(1,a_1),\ldots,(d,a_d)]$ as the positive orientation if $\sum_{i=1}^d a_i \equiv 1 \pmod 3$; negative if $\sum_{i=1}^d a_i \equiv 2 \pmod 3$.

Now, fix $d \geq 3$. Beginning with $M^{d-1}$, we construct a normal $(d-1)$-pseudomanifold $X$ by a stellar subdivision of the facets $\{(1,1),(2,0),(3,0),\ldots,(d,0)\}$ and $\{(1,0),(2,1),(3,0),\ldots,(d,0)\}$. We call the new vertices $v_1$ and $v_2$, respectively.

Beginning again with $M^{d-1}$, we similarly construct a normal $(d-1)$-pseudomanifold $Y$ by a stellar subdivision of the facets $\{(1,1),(2,0),(3,0),\ldots,(d,0)\}$ and $\{(1,0),(2,2),(3,0),\ldots,(d,0)\}$. We call the new vertices $w_1$ and $w_2$, respectively. Figure \ref{designs} shows the structures of $X$ and $Y$ for $d=3$.

\begin{figure}
    \centering
    \includegraphics[width=\textwidth]{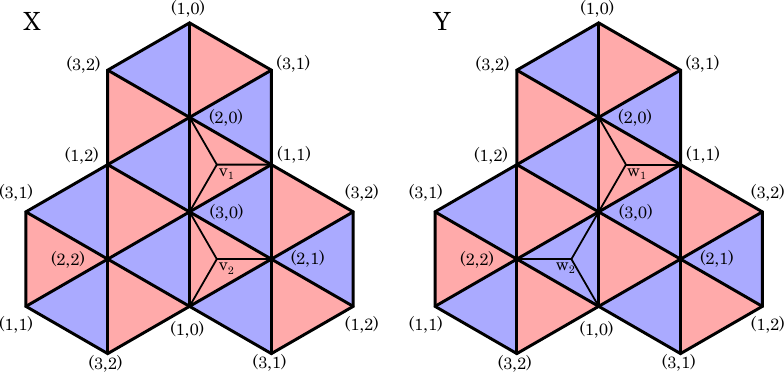}
    \caption{Normal $(d-1)$-pseudomanifolds $X$ and $Y$ for $d=3$. In this case, $X$ and $Y$ are manifolds homeomorphic to the 2-torus.}
    \label{designs}
\end{figure}

\begin{claim}
    $\mathscr{F}(X)$ is not isomorphic to $\mathscr{F}(Y)$.
\end{claim}

\begin{proof}
    Let $E_1,E_2 \in \mathscr{F}^1(X)$ be the edges
    \begin{align*}
        E_1 &= \operatorname{Lk}(v_1,X)\backslash\operatorname{Lk}(v_2,X) = \{(1,1),(2,0)\},\\
        E_2 &= \operatorname{Lk}(v_2,X)\backslash\operatorname{Lk}(v_1,X) = \{(1,0),(2,1)\}.
    \end{align*}
    Similarly, let $D_1,D_2 \in \mathscr{F}^1(Y)$ be the edges
    \begin{align*}
        D_1 &= \operatorname{Lk}(w_1,Y)\backslash\operatorname{Lk}(w_2,Y) = \{(1,1),(2,0)\},\\
        D_2 &= \operatorname{Lk}(w_2,Y)\backslash\operatorname{Lk}(w_1,Y) = \{(1,0),(2,2)\}.
    \end{align*}
    (We write $S \backslash T$, as above, to denote the subcomplex of $S$ consisting of all faces disjoint from $T$.)
    
    By way of contradiction, suppose there is an isomorphism $\varphi:\mathscr{F}(X) \to \mathscr{F}(Y)$. Observe that $v_1,v_2$ are the only vertices of $X$ contained in just $d$ facets, while $w_1,w_2$ are the only such vertices of $Y$. Thus, $\varphi\{v_1,v_2\}=\{w_1,w_2\}$. It follows that
    \begin{align*}
        \varphi\{\operatorname{Lk}(v_1,X)\backslash\operatorname{Lk}(v_2,X), \operatorname{Lk}(v_2,X)\backslash\operatorname{Lk}(v_1,X)\} &= \{\operatorname{Lk}(w_1,Y)\backslash\operatorname{Lk}(w_2,Y), \operatorname{Lk}(w_2,Y)\backslash\operatorname{Lk}(w_1,Y)\}\\
        \Rightarrow \varphi\{E_1,E_2\} &= \{D_1,D_2\}.
    \end{align*}
    This is impossible, because
    \[
    \operatorname{Lk}(E_1,X)\backslash\{v_1,v_2\} = \operatorname{Lk}(E_2,X)\backslash\{v_1,v_2\},
    \]
    while
    \[
    \operatorname{Lk}(D_1,Y)\backslash\{w_1,w_2\} \neq \operatorname{Lk}(D_2,Y)\backslash\{w_1,w_2\}.
    \]
    Hence, by contradiction, $\mathscr{F}(X)$ is not isomorphic to $\mathscr{F}(Y)$.
\end{proof}

\begin{claim}
    $\mathscr{F}_0^{d-2}(X)$ is isomorphic to $\mathscr{F}_0^{d-2}(Y)$.
\end{claim}

\begin{proof}
    Let $\varphi:\mathscr{F}_0^{d-2}(X) \to \mathscr{F}_0^{d-2}(Y)$ be the map taking $v_1,v_2,(2,1),(2,2)$ to $w_1,w_2,(2,2),(2,1)$, respectively, and fixing all other vertices. Then $\varphi$ permutes the faces of $\mathscr{F}_0^{d-2}(M^{d-1}) \cong \mathscr{F}_0^{d-2}(X\backslash\{v_1,v_2\}) \cong \mathscr{F}_0^{d-2}(Y\backslash\{w_1,w_2\})$. We can also see that $\varphi(\operatorname{Lk}(v_1,X)) = \operatorname{Lk}(w_1,Y)$ and $\varphi(\operatorname{Lk}(v_2,X)) = \operatorname{Lk}(w_2,Y)$. Thus, $\varphi:\mathscr{F}_0^{d-2}(X) \to \mathscr{F}_0^{d-2}(Y)$ is an isomorphism.
\end{proof}

In conclusion,

\begin{theorem}
\label{design-theorem}
    For all $d \geq 3$, there exist normal, orientable $(d-1)$-pseudomanifolds $X$ and $Y$ such that
    \begin{align*}
        \mathscr{F}_0^{d-2}(X) &\cong \mathscr{F}_0^{d-2}(Y),\\
        \mathscr{F}(X) &\not\cong \mathscr{F}(Y).
    \end{align*}
\end{theorem}

\section{Acknowledgments}
The author would like to thank Jos\'e Samper for bringing these problems to his attention. The author would also like to thank Isabella Novik for her invaluable guidance in every step of this paper's creation. The author was  partially supported by a graduate fellowship from NSF grant DMS-2246399.

\bibliography{bibliography}
\bibliographystyle{plain}

\end{document}